\documentclass[10pt]{article}
\usepackage{amsmath,amsthm,amsfonts,amssymb,latexsym,url}
\usepackage[T2A]{fontenc}
\usepackage[utf8]{inputenc}

\allowdisplaybreaks[4]

\newcommand{\Extra}[1]{}

\newtheorem{lemma}{Lemma}
\newtheorem*{corollary}{Corollary}
\newtheorem{theorem}{Theorem}
\theoremstyle{definition}

\DeclareMathOperator{\DBern}{D^{\mathrm{Bernoul}}}            
\DeclareMathOperator{\Dbin}{D^{\mathrm{binom}}_{\mathit{n}}}  
\DeclareMathOperator{\Dexch}{D^{\mathrm{exch}}}               

\newcommand*{\bin}{\mathrm{bin}} 
\newcommand*{\FFF}{\mathcal{F}}  
\newcommand*{\KP}{\mathit{KP}}   

\title{On the concept of Bernoulliness}
\author{Vladimir Vovk}

\begin{document}
\maketitle

\begin{abstract}
  The first part of this paper is another English translation of \cite{Vovk:1986}.
  It gives a natural definition of a finite Bernoulli sequence
  (i.e., a typical realization of a finite sequence of binary IID trials)
  and compares it with the Kolmogorov--Martin-L\"of definition,
  which is interpreted as defining exchangeable sequences.
  The appendix gives the historical background and proofs.

  The version of this paper at \url{http://alrw.net} (Working Paper 15)
  is updated more often than the arXiv report.
\end{abstract}

This note gives a definition of a ``Bernoulli sequence'',
i.e., a finite sequence of 0s and 1s that is random with respect to the class of Bernoulli measures.
Our definition is different from A.~N.~Kolmogorov's \cite{Kolmogorov:1968}
and more similar to the definition in \cite{Levin:1973};
for terminological convenience, sequences random in the sense of \cite{Kolmogorov:1968}
will be called collectives.

\paragraph{1.}
Denote by $X$ the union of the following sets:
$2^*$, the set of all finite sequences of 0s and 1s,
$N^*$, the set of all finite sequences of natural numbers,
$N$, the set of all natural numbers,
and $\FFF(N)$, the set of all finite subsets of the set $N$
(a finite subset of $N$ is identified with the list of its elements in increasing order).
The set $X$ defined in this way contains all objects that we will need.

A \emph{description method} is a partial function $B:2^*\times X\to X$ that has an algorithm computing its values
(the reader not familiar with the theory of algorithms
can safely rely on his intuitive idea of computability).
The length of a shortest $p$ such that $B(p,y)=x$ will be denoted $K_B(x\mid y)$
and called the \emph{complexity of $x$ given $y$ under the description method $B$}.

\begin{lemma}\label{lem:1}
  There exists a description method $A(p,y)$ such that, for any description method $B(p,y)$,
  \[
    K_A(x\mid y) \le K_B(x\mid y) + C,
  \]
  where $C$ is a constant that does not depend on $x$ and $y$.
\end{lemma}

A description method is \emph{prefix} if, for any $p\in2^*$ and $p'\in2^*$ such that $p$ is a prefix of $p'$,
$B(p,y)=B(p',y)$ for all $y\in X$.
Lemma~\ref{lem:1} will remain true if ``description method'' is replaced by ``prefix description method''.

Let us fix a description method $A$ satisfying Lemma~\ref{lem:1};
$K_A(x\mid y)$ will be denoted $K(x\mid y)$ and called the \emph{complexity of $x$ given $y$}.
The \emph{prefix complexity $\KP(x\mid Y)$ of $x$ given $y$} is defined similarly.
Proofs of the assertions made above can be found in \cite{Vyugin:1994}.

\paragraph{2.}
Denote by $2^{(n)}$ the set of all sequences in $2^*$ of length $n$.
Let $p\in[0,1]$ and $n>0$ be an integer.
On the set $2^{(n)}$ define the Bernoulli measure with parameters $(n,p)$ as follows:
for any $x\in2^{(n)}$ set
$P\{x\}=p^k(1-p)^{n-k}$,
where $k$ is the number of 1s in $x$.
On the set $\{0,1,\ldots,n\}$ define the binomial measure with parameters $(n,p)$
by the equality
$P\{k\}=\binom{n}{k}p^k(1-p)^{n-k}$
for all $k=0,1,\ldots,n$.

With each Bernoulli (binomial) measure $P$ with parameters $(n,p)$
associate an integer-valued function $T(x\mid P)$ of the variable $x\in2^{(n)}$
(in the case of a binomial measure, $x\in\{0,1,\ldots,n\}$) so that:
\begin{itemize}
\item[(1)]
  $E2^{T(x\mid P)}\le1$, where $E$ stands for the mean under the measure $P$.
\item[(2)]
  As function of $x$ and $P$,
  the function $T$ is lower semicomputable.
  This means that there exists an algorithm $A$ such that:
  given $n$, $x$, and an ``oracle'' that for each $i\in N$ outputs a rational number $a_i$ satisfying
  $\left|a_i-p\right|\le2^{-i}$
  (of course, the sequence $a_i$ does not have to be computable),
  $A$ enumerates a nondecreasing sequence of integer numbers $m_j$
  such that $\sup_j m_j = T(x\mid P)$;
  it is required that $A$ should work correctly for an arbitrary ``oracle''
  (for oracular computability, see \cite{Rogers:1967}).
\end{itemize}
Such functions $T$ will be called \emph{tests} (for randomness).
A condition similar to (1) first appeared in \cite{Levin:1976}.

\begin{lemma}\label{lem:2}
  There exists a test $D(x\mid P)$ such that, for any test $T(x\mid P)$,
  \[
    D(x\mid P) \ge T(x\mid P) - C,
  \]
  where $C$ does not depend on $x$ and $P$.
\end{lemma}

The test $D(x\mid P)$ in Lemma~\ref{lem:2} will be called the \emph{randomness deficiency of $x$ with respect to the measure $P$}.

\paragraph{3.}
The \emph{Bernoulliness deficiency} $\DBern(x)$ of a sequence $x\in2^{(n)}$ is defined to be $\inf D(x\mid P)$,
where the $\inf$ is over the Bernoulli measures with parameters $(n,p)$ for all $p\in[0,1]$.
Similarly,
for a number $k$ in $\{0,1,\ldots,n\}$ define the \emph{binomiality deficiency} $\Dbin(k)$
as $\inf D(k\mid P)$;
here the $\inf$ is over the binomial measures.
A sequence in $2^*$ is called \emph{Bernoulli} if its Bernoulliness deficiency is small.
We speak of a binomial number in a similar sense.

\begin{theorem}\label{thm:1}
  If $x\in2^{(n)}$ contains $k$ 1s, then
  \begin{equation}\label{eq:thm1}
    \DBern(x)
    -
    \left[
      \log_2\binom{n}{k} - \KP(x\mid n,k,\Dbin(k))
    \right]
    =
    \Dbin(k) + O(1).
  \end{equation}
\end{theorem}

According to \cite{Kolmogorov:1968},
the $x$ in \eqref{eq:thm1} 
is a collective if
$\log_2\binom{n}{k} - K(x\mid n,k)$
is small
(intuitively, this means that the complexity of $x$ in the class of sequences in $2^{(n)}$ containing $k$ 1s
is close to maximal).
The expression in square brackets in~\eqref{eq:thm1} is completely analogous to
$\log_2\binom{n}{k} - K(x\mid n,k)$,
except for the term $\Dbin(k)$.
We can get rid of it at the expense of a certain loss of sharpness.

\begin{corollary}
  For a fixed $\epsilon>0$,
  \begin{multline*}
    \Dbin(k) - O(1)
    \le
    \DBern(x)
    -
    \left[
      \log_2\binom{n}{k} - \KP(x\mid n,k)
    \right]\\
    \le
    (1+\epsilon)\Dbin(k) + O(1).
  \end{multline*}
\end{corollary}

The word ``fixed'' in the statement of the corollary means that $\left|O(1)\right|$
is bounded by a value that depends on $\epsilon$.
The quantities $\log_2\binom{n}{k} - \KP(x\mid n,k)$ and $\log_2\binom{n}{k} - K(x\mid n,k)$
differ by at most
\[
  2\log_2
  \left(
    \left|
      \log_2\binom{n}{k} - K(x\mid n,k)
    \right|
    +
    1
  \right)
  +
  O(1)
\]
(we refer to this as coincidence to within $2\log$).

Therefore, our definition of Bernoulliness adds to the requirement of nearly maximal complexity
of $x$ in the class of sequences in $2^*$ with the same length and the same number of 1s as $x$
the requirement of binomiality of the number of 1s.
It turns out that binomiality deficiency can be characterized in complexity-theoretic terms;
this gives a complexity-theoretic characterization of Bernoulli sequences.

If $\mathfrak{A}$ is a partition (of a set into disjoint subsets),
$\mathfrak{A}(k)$ denotes the element of the partition containing $k$.

\begin{theorem}\label{thm:2}
  Let $n>0$ be an integer.
  Set
  \begin{equation}\label{eq:thm2-partition}
    k_s
    =
    \frac{n}{2}
    \left(
      1
      -
      \cos\frac{s}{\sqrt{n}}
    \right)
    \text{\quad for\quad}
    s=0,1,\ldots,\lfloor\pi\sqrt{n}\rfloor,
  \end{equation}
  where $\lfloor\cdot\rfloor$ is integer part.
  Denote by $\mathfrak{A}$ the partition of the set $\{0,1,\ldots,n\}$ into the subsets
  $[k_s,k_{s+1})$,
  where $s=0,1,\ldots,\lfloor\pi\sqrt{n}\rfloor$
  (for $s=\lfloor\pi\sqrt{n}\rfloor$ we set $k_{s+1}=+\infty$).
  If $k\in\{0,1,\ldots,n\}$,
  \begin{equation}\label{eq:Dbin}
    \Dbin(k)
    =
    \log_2\left|\mathfrak{A}\right|
    -
    \KP(k\mid n,\mathfrak{A}(k))
    +
    O(1).
  \end{equation}
\end{theorem}

The right-hand side of \eqref{eq:Dbin} 
coincides with
$
  \log_2\left|\mathfrak{A}\right|
  -
  K(k\mid n,\mathfrak{A}(k))
$
to within $2\log$.
Notice the following properties of the partition $\mathfrak{A}$.
The sets $\{0\}$ and $\{n\}$ are in $\mathfrak{A}$.
If $k\ne0$ and $k\ne n$,
\[
  \left|\mathfrak{A}(k)\right|
  =
  \sqrt{\frac{k(n-k)}{n}}
  \cdot
  2^{O(1)}.
\]
It is easy to see that
\[
  \sqrt{\frac{k(n-k)}{n}}
  =
  \sqrt{n\frac{k}{n}\left(1-\frac{k}{n}\right)}
\]
is an estimate of the standard deviation of the number of 1s.

The author is deeply grateful to his supervisor A.~N.~Kolmogorov
for valuable discussions.
V.~V.~V'yugin's and A.~K.~Zvonkin's comments contributed to the improvement of this note,
and the author expresses his sincere gratitude to them.

\noindent
Scientific Council for Cybernetics \hfill Received by the Board of Governors\\
Academy of Sciences of the USSR \hfill on 6 April 1984\\

\appendix
\section{Historical background}

Kolmogorov's definition of Bernoulliness was part of his project of creating new mathematical foundations
for applications of probability.
By that time, the measure-theoretic foundations for the theory of probability
clearly articulated in his \emph{Grundbegriffe} \cite{Kolmogorov:1933}
had been accepted by the community of researchers working in mathematical probability and statistics
(see, e.g., \cite{Shafer/Vovk:2006};
an important role in the acceptance belonged to Doob's 1953 book \cite{Doob:1953}).
In his approach to the foundations of applications of probability,
Kolmogorov followed Richard von Mises
(referring to \cite{Mises:1931} in Section~I.2 of \cite{Kolmogorov:1933}).
Richard von Mises suggested frequentist foundations based on the notion of a collective;
his original definition was not rigorous, but two different formalizations
were suggested by Abraham Wald \cite{Wald:1937} and Alonzo Church \cite{Church:1940}.
Collectives are often regarded as the first attempt to define the notion of a random sequence
(roughly, a sequence that is a typical realization of a sequence of independent and identically distributed trials).

Kolmogorov regarded collectives as important only for foundations of applications of probability.
He believed that there was no need to change the existing foundations of the theory of probability
based on Kolmogorov's \cite{Kolmogorov:1933} axioms:
``there is no need whatsoever to change the established construction of the mathematical probability theory
on the basis on the general theory of measure'' (\cite{Kolmogorov:1983}, Section~6).

Kolmogorov's first attempt to bring von Mises's infinitary notion of collectives
closer to the needs of practice was his 1963 paper \cite{Kolmogorov:1963}.
After introducing his algorithmic notion of complexity in 1965 \cite{Kolmogorov:1965},
Kolmogorov used it in \cite{Kolmogorov:1968} to define a finitary notion of a binary collective
that was in some sense universal (being based on a universal notion of algorithmic complexity)
and so immune to known examples showing the inadequacy of collectives as formalization of random sequences
(such as Ville's \cite{Ville:1939} demonstration that collectives do not necessarily satisfy the conclusion
of the law of the iterated logarithm).
The same definition was published earlier by Per Martin-L\"of (\cite{Martin-Lof:1966}, Section~VI),
Kolmogorov's PhD student.

Kolmogorov and Martin-L\"of used the term ``Bernoulli sequences'' for their formalization of random sequences,
since they considered only the case of binary random sequences,
where the goal is to formalize typical realizations of repeated independent Bernoulli trials.
(Kolmogorov was very much against what he regarded as premature generalizations
and insisted that the simple binary case should be understood first.)

Two important features of Kolmogorov's approach to foundations of applications of probability
were its finitary character
(``we do not often see infinite sequences around us, do we?'')\
and avoiding probability measures when defining random sequences.
(Perhaps probability measures were to reappear at a later stage as frequencies in random sequences,
as in von Mises's writings.)
However, his PhD students, first of all Per Martin-L\"of \cite{Martin-Lof:1966} and Leonid Levin \cite{Levin:1973,Levin:1976}
were quick to bring into Kolmogorov's theory
both aspects that Kolmogorov himself avoided
(infinite sequences and probability measures).
In terms of probability measures,
Kolmogorov's preferred notion of randomness could be expressed as randomness
with respect to uniform probability measures on finite sets.

\section{This note}

Note \cite{Vovk:1986} defined Bernoulliness as randomness with respect to the class of Bernoulli measures
and compared the resulting notion with that of Kolmogorov \cite{Kolmogorov:1968} and Martin-L\"of \cite{Martin-Lof:1966}.
The latter was identified with the randomness with respect to the exchangeable distributions \eqref{eq:exch}
(without mentioning them explicitly).
The first main result (Theorem~1) of \cite{Vovk:1986} was that the Bernoulliness deficiency
decomposes into the sum of the exchangeability deficiency and the binomiality deficiency of the number of 1s,
where the binomiality deficiency is defined to be the randomness with respect to the class of binomial measures.
The second main result (Theorem~2) expressed the binomiality deficiency of a number $k$ in Kolmogorov's preferred terms,
as the randomness deficiency of $k$ with respect to the uniform probability measure on a certain neighbourhood of $k$.
Therefore, the new notion of Bernoulliness as a whole was expressed in Kolmogorov's preferred terms.

Note \cite{Vovk:1986} was published in a Russian journal
routinely translated into English cover-to-cover.
The current translation uses the one in \cite{Vovk:1986}
but follows the Russian original somewhat less closely.
In particular, it sets the terms being defined (such as ``description method'', ``complexity'', etc.)\
in italics.

\section{Proofs}

The proofs of Theorems~1 and~2 in \cite{Vovk:1986} have never been published,
but they are not difficult to extract from the existing publications,
such as \cite{Vovk/Vyugin:1993} (Theorem~1) and \cite{Vovk:1997} (Lemmas~1--3).
This section will spell them out.

\subsection*{Proof of Theorem~\protect\ref{thm:1}}

Let us set, for $x\in 2^{(n)}$ and $y\in X$,
\begin{equation}\label{eq:exch}
  \Dexch(x\mid y)
  :=
  \log_2\binom{n}{k} - \KP(x\mid n,k,y),
\end{equation}
where $k$ is the number of 1s in $x$.
This is the exchangeability deficiency of $x$
(the randomness deficiency $\inf_P D(x\mid P;y)$ with respect to all exchangeable distributions $P$ on $2^{(n)}$,
conditioned on knowing $y$).
In terms of $\Dexch$, our goal \eqref{eq:thm1} can be rewritten as
\begin{equation}\label{eq:goal}
  \DBern(x)
  =
  \Dexch(x\mid\Dbin(k))
  +
  \Dbin(k)
  +
  O(1).
\end{equation}

According to Theorem~1 in \cite{Vovk/Vyugin:1993},
we have
\begin{equation}\label{eq:basis}
  D(x\mid B_{n,p})
  =
  D(k\mid\bin_{n,p})
  +
  D(x\mid 2^n_k;D(k\mid\bin_{n,p}))
  +
  O(1)
\end{equation}
where
$B_{n,p}$ is the Bernoulli measure with parameters $(n,p)$,
$\bin_{n,p}$ is the binomial measure with parameters $(n,p)$,
$k$ is the number of 1s in $x$,
$2^n_k$ is the set of all sequences in $2^{(n)}$ with $k$ 1s
(identified with the uniform probability measure on this set),
and $D(x\mid 2^n_k;D(k\mid\bin_{n,p}))$ stands for the randomness deficiency of $x$ in $2^n_k$
conditioned on knowing $D(k\mid\bin_{n,p})$.
(Namely, \eqref{eq:basis} is obtained from the equation in Theorem~1 in \cite{Vovk/Vyugin:1993}
by applying $\log$ to both sides;
this is needed since the exposition in \cite{Vovk/Vyugin:1993} is in terms of ``level of impossibility'' $2^{-D}$,
where $D$ is randomness deficiency.)

Roughly, \eqref{eq:goal} corresponds to minimizing both sides of~\eqref{eq:basis} over $p\in[0,1]$;
this works because of the following theorem
(which is a version of Theorem~2 in \cite{Vovk:1997}).

\begin{theorem}\label{thm:core}
  There exists a computable point estimator $E:2^*\to[0,1]$ such that
  \begin{equation}\label{eq:core-1}
    \DBern(x)
    =
    D(x\mid B_{n,E(x)})
    +
    O(1)
  \end{equation}
  and
  \begin{equation}\label{eq:core-2}
    \Dbin(k)
    =
    D(k\mid\bin_{n,E(x)})
    +
    O(1),
  \end{equation}
  where $x$ ranges over $2^*$,
  $n$ is the length of $x$,
  and $k$ is the number of 1s in $x$.
\end{theorem}

An estimator $E$ satisfying the conditions in Theorem~\ref{thm:core}
is described at the beginning of Subsection~4.1 of \cite{Vovk:1997}.
For agreement with that paper,
let us replace the partition \eqref{eq:thm2-partition}
by the equivalent partition
(cf.\ the identity $1-\cos(2\alpha)=2\sin^2\alpha$ and Lemma~1 in \cite{Vovk:1997})
\begin{equation*}
  \theta_a
  =
  n
  \sin^2\frac{a}{\sqrt{n}}
  \text{\quad for\quad}
  a=0,1,\ldots,\lfloor\pi\sqrt{n}\rfloor.
\end{equation*}
It will be convenient (as in \cite{Vovk:1997}) to allow $a$ to be any number in the interval
$
  \left[
    0,
    \pi\sqrt{n}
  \right]
$.

Theorem~\ref{thm:core} and~\eqref{eq:basis} immediately imply~\eqref{eq:goal}:
\begin{multline*}
  \Dexch(x\mid\Dbin(k))
  +
  \Dbin(k)\\
  =
  D(x\mid 2^n_k;D(k\mid\bin_{n,E(x)}))
  +
  D(k\mid\bin_{n,E(x)})
  +
  O(1)\\
  =
  D(x\mid B_{n,E(x)})
  +
  O(1)
  =
  \DBern(x)
  +
  O(1).
\end{multline*}

\begin{proof}[Proof of Theorem~\ref{thm:core}]
  Let us check \eqref{eq:core-1};
  the proof of \eqref{eq:core-2} is similar.
  We are required to prove
  \begin{equation*}
    D(x\mid B_{n,p})
    \ge
    D(x\mid B_{n,E(x)})
    -
    O(1).
  \end{equation*}
  We will do this separately for the cases $\left|a-\hat a(x)\right|<1$ and $\left|a-\hat a(x)\right|\ge1$,
  where $a$ is defined by $\theta_a=p$
  and $\hat a$ is defined before Lemma~2 in \cite{Vovk:1997}.

  If $\left|a-\hat a(x)\right|<1$,
  \begin{align*}
    D(x\mid B_{n,p})
    &=
    -\log_2 B_{n,p}
    -
    \KP(x\mid n,p)
    +
    O(1)\\
    &\ge
    -\log_2 B_{n,E(x)}
    -
    \KP(x\mid n,E(x))
    -
    O(1)\\
    &=
    D(x\mid B_{n,E(x)})
    -
    O(1),
  \end{align*}
  where the inequality uses Lemma~2 in \cite{Vovk:1997}
  (ensuring $-\log_2 B_{n,p}\ge-\log_2 B_{n,E(x)}-O(1)$)
  and $\KP(x\mid n,E(x))\ge\KP(x\mid n,p)-O(1)$.

  If $\left|a-\hat a(x)\right|\ge1$,
  \begin{align*}
    D(x\mid B_{n,p})
    &=
    -\log_2 B_{n,p}
    -
    \KP(x\mid n,p)
    +
    O(1)\\
    &\ge
    -\log_2 B_{n,E(x)}
    +
    \epsilon
    \left|
      a - \hat a(x)
    \right|\\
    &\quad{}-
    \KP(x\mid n,E(x))
    -
    2\log
    \left(
      \left|
        a - \hat a(x)
      \right|
      +
      1
    \right)
    -
    O(1)\\
    &\ge
    D(x\mid B_{n,E(x)})
    -
    O(1),
  \end{align*}
  where $\epsilon>0$ is a universal constant
  and the inequality uses Lemma~3 in \cite{Vovk:1997}
  and the standard bound $\KP(m)\le2\log m+O(1)$ for $m\in\{1,2,\ldots\}$.
\end{proof}

\subsection*{Proof of Theorem~\protect\ref{thm:2}}

Theorem~\ref{thm:2} will follow from \eqref{eq:core-2} in Theorem~\ref{thm:core}
and Lemma~2 in \cite{Vovk:1997}
if we show that
\[
  B_p
  \left\{
    \left|a-\hat a(x)\right|<1
  \right\}
  \ge
  \epsilon
\]
for some universal constant $\epsilon>0$,
where again $a$ and $p$ are connected by $\theta_a=p$.
This follows immediately from Corollary~3 of \cite{Vovk:1997}
and Chebyshev's inequality.

\renewcommand\refname{Additional references}


\begin{thebibliography}{9}
\bibitem{Kolmogorov:1968}
  Andrei N. Kolmogorov.
  Logical basis for information theory and probability theory.
  \emph{IEEE Transactions of Information Theory}, IT-14:662--664, 1968.
  Russian original: К логическим основам теории информации и теории вероятностей
    (published in 1969 in \emph{Проблемы передачи информации}).
\bibitem{Levin:1973}
  Leonid A. Levin.
  On the notion of a random sequence.
  \emph{Soviet Mathematics Doklady}, 14:1413--1416, 1973.
  Russian original:
    Л.~А.~Левин.
    О понятии случайности.
    \emph{Доклады АН СССР} 212(1):548--550, 1973.
\bibitem{Vyugin:1994}
  Vladimir V. V'yugin.
  Algorithmic entropy (complexity) of finite objects and its applications
    to defining randomness and amount of information.
  \emph{Selecta Mathematica Sovietica}, 13:357--389, 1994.
  Russian original:
    В.~В.~Вьюгин.
    Алгоритмическая энтропия (сложность) конечных объектов и ее применение
    к определению случайности и количества информации.
    \emph{Семиотика и информатика} 16:14--43, 1981. 
\bibitem{Rogers:1967}
  Hartley Rogers, Jr.
  \emph{Theory of Recursive Functions and Effective Computability}.
  McGraw-Hill, New York, 1967.
\bibitem{Levin:1976}
  Leonid A. Levin.
  Uniform tests of randomness.
  \emph{Soviet Mathematics Doklady}, 17:337--340, 1976.
  Russian original:
    Л.~А.~Левин.
    Равномерные тесты случайности.
    \emph{Доклады АН СССР} 227(1):33--35, 1976.
\end{thebibliography}

\begin{thebibliography}{19}
\makeatletter
\addtocounter{\@listctr}{5}
\makeatother
\bibitem{Church:1940}
  Alonzo Church.
  On the concept of a random sequence.
  \emph{Bulletin of American Mathematical Society}, 46:130--135, 1940.
\bibitem{Doob:1953}
  Joseph L. Doob.
  \emph{Stochastic Processes}.
  Wiley, New York, 1953.
\bibitem{Kolmogorov:1933}
  Andrei N. Kolmogorov.
  \emph{Grundbegriffe der Wahrscheinlichkeitsrechnung}.
  Springer, Berlin, 1933.
  English translation: Foundations of the Theory of Probability.
  Chelsea, New York, 1950.
\bibitem{Kolmogorov:1963}
  Andrei N. Kolmogorov.
  On tables of random numbers.
  \emph{Sankhya. Indian Journal of Statistics A}, 25:369--376, 1963.
\bibitem{Kolmogorov:1965}
  Andrei N. Kolmogorov.
  Three approaches to the quantitative definition of information.
  \emph{Problems of Information Transmission}, 1:1--7, 1965.
  Russian original: Три подхода к определению понятия ``количество информации''.
\bibitem{Kolmogorov:1983}
  Andrei N. Kolmogorov.
  Combinatorial foundations of information theory and the calculus of probabilities.
  \emph{Russian Mathematical Surveys}, 38:29--40, 1983.
  Russian original: Комбинаторные основания теории информации и исчисления вероятностей.
\bibitem{Martin-Lof:1966}
  Per Martin-L\"of.
  The definition of random sequences.
  \emph{Information and Control}, 9:602--619, 1966.
\bibitem{Mises:1931}
  Richard von Mises.
  \emph{Vorlesungen aus dem Gebiete der angewandten Mathematik.
  I. Band.
  Wahrscheinlichkeitsrechnung und ihre Anwendung in der Statistik und theoretischen Physik}
  (Lectures on Applied Mathematics. Vol. 1. Probabilities and their Applications in Statistics and Theoretical Physics).
  Franz Deuticke, Leipzig and Vienna, 1931.
\bibitem{Shafer/Vovk:2006}
  Glenn Shafer and Vladimir Vovk.
  The origins and legacy of Kolmogorov's \emph{Grundbegriffe}.
  The Game-The\-o\-ret\-ic Probability and Finance project,
  \texttt{http://prob\linebreak[0]a\linebreak[0]bil\linebreak[0]i\linebreak[0]ty\linebreak[0]and\linebreak[0]fi\linebreak[0]nance.com},
  Working Paper 4,
  April 2013.
  Part of this technical report
  (covering the \emph{Grundbegriffe} and the period before its publication)
  appeared as:
  The sources of Kolmogorov's \emph{Grundbegriffe}.
  \emph{Statistical Science}, 21:70--98, 2006.
\bibitem{Ville:1939}
  Jean Ville.
  \emph{Etude critique de la notion de collectif}.
  Gauthier-Villars, Paris, 1939.
\bibitem{Vovk:1986}
  Vladimir Vovk.
  On the concept of the Bernoulli property.
  \emph{Russian Mathematical Surveys}, 41:247--248, 1986.
  Russian original:
    В.~Г. Вовк.
    О понятии бернуллиевости.
    \emph{Успехи математических наук},
      41(1):185--186, 1986.
\bibitem{Vovk:1997}
  Vladimir Vovk.
  Learning about the parameter of the Bernoulli model.
  \emph{Journal of Computer and System Sciences}, 55:96--104, 1997.
\bibitem{Vovk/Vyugin:1993}
  Vladimir Vovk and Vladimir V. V'yugin.
  On the empirical validity of the Bayesian method.
  \emph{Journal of the Royal Statistical Society B}, 55:253--266, 1993.
\bibitem{Wald:1937}
  Abraham Wald.
  Die Widerspruchfreiheit des Kollectivbegriffes der Wahrscheinlichkeitsrechnung.
  \emph{Ergebnisse eines Mathematischen Kolloquiums}, 8:38--72, 1937.
\end{thebibliography}
\end{document}